\documentclass[12pt]{article}

\usepackage{graphicx}
\usepackage{nameref}
\usepackage[shortlabels,inline]{enumitem}
\usepackage{empheq}
\usepackage[shortlabels]{enumitem}
\setlist[enumerate]{nosep}
\usepackage[doc]{optional}
\usepackage{xcolor}
\usepackage[colorlinks=true,
            linkcolor=refkey,
            urlcolor=lblue,
            citecolor=red]{hyperref}
\usepackage{float}
\usepackage{soul}
\usepackage{graphicx}
\definecolor{labelkey}{rgb}{0,0.08,0.45}
\definecolor{refkey}{rgb}{0,0.6,0.0}
\definecolor{Brown}{rgb}{0.45,0.0,0.05}
\definecolor{lime}{rgb}{0.00,0.8,0.0}
\definecolor{lblue}{rgb}{0.5,0.5,0.99}

 \usepackage{mathpazo}

\colorlet{hlcyan}{cyan!30}

\usepackage{stmaryrd}
\usepackage{amssymb}
\makeatletter
\def\namedlabel#1#2{\begingroup
   \def\@currentlabel{#2}%
   \label{#1}\endgroup
}
\makeatother

\newcommand{\seppthree}{\setlength{\itemsep}{-3pt}}

\newcommand{\To}{\ensuremath{\rightrightarrows}}

\newcommand{\fenv}[1]%
{\ensuremath{\,\overrightarrow{\operatorname{env}}_{#1}}}
\newcommand{\benv}[1]%
{\ensuremath{\,\overleftarrow{\operatorname{env}}_{#1}}}

\newcommand{\scal}[2]{\left\langle{#1},{#2}  \right\rangle}

\newcommand{\RR}{\ensuremath{\mathbb R}}

\newcommand{\NN}{\ensuremath{\mathbb N}}

\newcommand{\dom}{\ensuremath{\operatorname{dom}}}

\newcommand{\ran}{\ensuremath{{\operatorname{ran}}\,}}

\newcommand{\Fix}{\ensuremath{\operatorname{Fix}}}
\newcommand{\Id}{\ensuremath{\operatorname{Id}}}

{\begin{list}{}{%
\settowidth{\labelwidth}{\textrm{#1~}}%
\setlength{\leftmargin}{\labelwidth+\labelsep}}}
{\end{list}}
\usepackage{amsthm}
\usepackage[capitalize,nameinlink]{cleveref}
\crefname{equation}{}{equations}
\crefname{chapter}{Appendix}{chapters}
\crefname{item}{}{items}
\crefname{enumi}{}{}
\newtheorem{theorem}{Theorem}[section]
\newtheorem{lemma}[theorem]{Lemma}

\newtheorem{proposition}[theorem]{Proposition}

\newtheorem{example}[theorem]{Example}
\newtheorem{ex}[theorem]{Example}
\newtheorem{fact}[theorem]{Fact}
\newtheorem{remark}[theorem]{Remark}

\providecommand{\norm}[1]{\lVert#1\rVert}

\providecommand{\eps}{\varepsilon}

\providecommand{\RR}{\mathbb{R}}

\providecommand{\ran}{\operatorname{ran}}

\providecommand{\dom}{\operatorname{dom}}

\newcommand{\fix}{\ensuremath{\operatorname{Fix}}}

\providecommand{\gra}{\operatorname{gra}}
\providecommand{\Id}{\operatorname{{ Id}}}

\providecommand{\To}{\rightrightarrows}

\providecommand{\NN}{\mathbb{N}}

\providecommand{\fix}{\operatorname{Fix}}
\providecommand{\ran}{\operatorname{ran}}

\providecommand{\Id}{\operatorname{Id}}
\providecommand{\PR}{\operatorname{P}}

\providecommand{\RR}{\mathbb{R}}
\providecommand{\NN}{\mathbb{N}}

\definecolor{myblue}{rgb}{.8, .8, 1}
  \newcommand*\mybluebox[1]{%
    \colorbox{myblue}{\hspace{1em}#1\hspace{1em}}}

\allowdisplaybreaks 

\begin{document}

%

\title{\textsc{Additional Studies on Displacement Mapping with Restrictions }}
\author{
Salihah Th.\ Alwadani\thanks{
Mathematics, Yanbu Industrial College, The Royal Comission for Jubail and Yanbu, 30436, Saudi Arabia. E-mail:
\texttt{salihah.s.alwadani@gmail.com}.}
}

\date{May 5, 2024}
\maketitle

\vskip 8mm

\begin{abstract} 
The theory of monotone operators plays a major role in modern optimization and many areas of nonlinera analysis. The central classes of monotone operators are matrices with a positive semidefinite symmetric part and subsifferential operators. In this paper, we complete our study to the displacement mappings. We derive formulas for set-valued and Moore-Penrose inverses. We also give a comprehensive study of the the operators ( $\big( 1 / 2 \big) \Id + T$ and its inverse) and provide a formula for  $ \big( \big( 1 / 2 \big) \Id + T\big)^{-1}$.We illustrate our results by considering the reflected  and the projection operators to closed linear subspaces.
\end{abstract}

{\small

}

\maketitle

\section{Introduction}

It is well known that  one of  important classes of monotone operators are Displacement mappings of nonexpansive mappings. There are many key examples that have proven  how these mappings are highly useful in optimization problems. For example, in 2016 Heiz Bauschke, Warren Hare, and Wala Moursi used displacement mappings in  analyzing  the range of the Douglas–Rachford operator to derive valuable duality results, see \cite{bauschke2016range}. Additionally, the asymptotic regularity results for nonexpansive mappings were generalized in \cite{bauschke2018magnitude} to the broader context of displacement mappings.
Overall, the displacement mapping framework has emerged as a powerful tool for analyzing the behavior of nonexpansive mappings, with a range of important applications in optimization and related areas.  Throughout, we assume that 
\begin{empheq}[box=\mybluebox]{equation}
\text{$X$ is
a real Hilbert space with inner product
$\scal{\cdot}{\cdot}\colon X\times X\to\RR$, }
 \end{empheq}
and induced norm $\|\cdot\|\colon X\to\RR\colon x\mapsto \sqrt{\scal{x}{x}}$. We also assume that $A: X \To X$ and $B: X \To X$ are maximally monotone operators. The \emph{resolvent} and the \emph{reflected resolvent} associated with $A$ are
\begin{equation}\label{quli1}
J_{A} = (\Id + A)^{-1} \ \ \text{and} \ \ R_{A} = 2 J_{A} - \Id,
\end{equation}
 respectively.  An operator $T: X \To X$ is \emph{nonexpansive} if it is Lipschitz continuous with constant $1$, i.e.,
 \begin{equation}\label{footnote11}
 \big(  \forall x \in X\big) \big( \forall y \in X \big)\norm{Tx - Ty} \leq \norm{x-y}.
 \end{equation}
Moreover, $T: D \To X$ is \emph{firmly nonexpansive} if 
\begin{equation}
\big( \forall x \in D\big) \big( \forall y \in D\big) \ \ \norm{Tx - Ty}^2+ \norm{( \Id - T) x - ( \Id - T) y}^2 \leq \norm{x - y}^2.
\end{equation}
\begin{fact}\label{def1}\cite[Definition~4.10]{BC2017} Let $D$ be a nonempty subset of $X$, let $T: D \to X$, and let $\beta \in \RR_{++}$. Then $T$ is $\beta$-cocoercive ( or $\beta$-inverse strongly monotone) if $\beta T$ is firmly nonexpansive, i.e., 
	$$ \big( \forall x \in D\big) \big( \forall y \in D\big) \ \ \scal{x-y}{ Tx-Ty} \geq \beta \norm{ Tx - Ty}^2. $$ 
\end{fact}

In optimization, we have seen the importance the \emph{displacement mappings }of nonexpansive mappings: 
 \begin{empheq}[box=\mybluebox]{equation}
\text{$ \Id - R$ }
 \end{empheq}
because of the nice properities that have such as monotonicity which plays a central role in modern optimization (see \cite{BC2017,reginaset, rockafellar2009variational, simons2006minimax, simons2008hahn, zeidler2013nonlinear, zeidler2014nonlinear} for more details). A comprehensive analysis of the displacement mappings of nonexpansive mappings from the point of view of monotone operator theory under the condition of isometry of finite order of $R$ are given in \cite[Lemma]{Alwadanithesis} and \cite[Section~3]{alwadani2021resolvents}. We refer the reder to  \cite[Exercise 12.16]{rockafellar2009variational}, and \cite[Example 20.29]{BC2017}, \cite{bauschke2016order}. 
 
Throughout this paper, we assume that
 \begin{empheq}[box=\mybluebox]{equation}\label{Assme1}
\text{$R: X \to X$ is
linear and nonexpansive, with 
$D:= \fix R= \ker \big( \Id - R \big)$. }
 \end{empheq} 
\emph{In this paper, we study the displacement mapping using the assumption in \cref{Assme1}}. Our results can be summarized as follows
 \begin{itemize}
 \item \cref{propr}, \cref{More1}, and \cref{More2} collect some useful properities of the dispdisplacment mapping and its inverse, which will be useful in our study.
 \item  \cref{T-opr} provides a formula and gives nice properties of the operator $T$.
 \item We derive a formula for the inverse of the displacment mapping (see \cref{The11} \ref{The110}). A formula for the Moore-Penrose inverse of the displacement mapping is given in \cref{The11}\ref{The111}.
 \item \cref{InT} gives a comprehensive study of the the operators $\big( 1 / 2 \big) \Id + T$ and its inverse. Additionaly, we derive a formula of $ \Big( \big( 1 / 2 \big) \Id + T\Big)^{-1}$ and prove that is equal to the resolvant of the operator $2T$. 
 \item We illustrates the reults by giving four examples. The first two examples are related to the projection operator to a closed linear subspace (see \cref{examp1} and \cref{examp2}), while the other two are related to the reflected operator to closed linear subspace (see \cref{examp3} and \cref{examp4}).
 \end{itemize}

\section{Results}\label{Section2}

Important properties of the displacement mapping $( \Id - R)$ and its inverse are given in the next proposition. 
\begin{proposition}\label{propr} Let $R$ be nonexpansive operator, then the following hold:
	\begin{enumerate}
	\item\label{faccc1} $ \frac{1}{2} ( \Id - R)$ is firmly nonexpansive.
	\item\label{gdee} $\Id - R$ is nonexpansive.
	\item\label{fac2} $\Id - R$ and $( \Id - R)^{-1}$ are maximally monotone.
	\item\label{fac3} $\Id - R$ is $\frac{1}{2}$-cocoercive.
	\item\label{fac4}  $(\Id - R)^{-1}$ is strongly monotone\footnote{ An operator $A: X \To X$ is strongly monotone with constant $\beta \in \RR_{++}$ if $A- \beta \Id$ is montone, i.e., 
$$ \big( \forall (x, u) \in \gra A\big) \big( \forall (y, v) \in \gra A\big) \ \ \ \scal{x -y}{u - v} \geq \beta \norm{x - y}^2.$$ \label{footnote 1}
}with constant $\frac{1}{2}$.
	\item\label{fac5} $\Id - R$ is $3^{*}$ monotone.
	\item\label{fac6}  $(\Id - R)^{-1}$ is $3^{*}$ monotone
	\item\label{fac7} $\Id - R$ is paramonotone.
	\item\label{fac8} $\big( \Id - R\big)^{-1} - \frac{1}{2} \Id $ is maximally monotone.
	\end{enumerate}
\end{proposition}

\begin{proof} \ref{faccc1}: We have 
\begin{align*}
R \ \ \text{is nonexpansive} \  & \Leftrightarrow -R = 2 \Big(  \big(  \Id - R\big) / 2 \Big) \ \ \text{is nonexpansive} \\
& \Leftrightarrow \big(  \Id - R\big) / 2 \ \ \text{is firmly nonexpansive},
\end{align*}
by \cite[Proposition~4.4]{BC2017}.
\ref{gdee}:  It follows from \ref{faccc1} and \cite[Proposition~4.2]{BC2017} .
\cref{fac2}: See \cite[Example 25.20(v)]{BC2017} or \cite[Theorem~7.1]{Alwadanithesis}.
 \ref{fac3}: Combine \ref{faccc1} and \cref{def1}.
 \ref{fac4}: Take $(x, u) \in \gra (\Id - R)^{-1}$ and $(y, v) \in \gra (\Id - R)^{-1}$. Then $u \in ( \Id - R)^{-1} x \Rightarrow x = u - Ru$ and $v \in ( \Id - R)^{-1} y \Rightarrow y = v - Rv$. 
\begin{align*}
& \ \ \ \ \ \ \scal{ u - v}{x - y} \geq \frac{1}{2} \norm{x - y}^2 \\
& \Leftrightarrow \scal{ u - v}{(u - Ru) - (v - Rv)} \geq \frac{1}{2} \norm{(u - Ru) - (v - Rv)}^2,
\end{align*}
which deduce from \ref{fac3} and \cref{footnote 1} that  $(\Id - R)^{-1}$ is strongly monotone with constant $( 1/ 2)$.
\cref{fac5} and \ref{fac6}: It follows from \ref{fac3} that $\Id - R$ is bounded by $( 1 / 2)$ and its monotone by \ref{fac2}. Hence,  $\Id - R$ and $( \Id - R)^{-1}$ are $3^{*}$ monotone by \cite[Proposition~25.16(i)~\&~(iv)]{BC2017}.\\
\cref{fac7}: See \cite[Example~22.9]{BC2017}.
\ref{fac8}: By \ref{fac3} and \cite[Example~22.7]{BC2017}, $( \Id - R)^{-1}$ is $( 1/ 2)$-strongly monotone, i.e., $B:= ( \Id - R)^{-1} - \frac{1}{2} \Id$ is still monotone. If $B$ was not maximally monotone, then neither would be $ B+ \frac{1}{2} \Id = ( \Id - R)^{-1}$ which would contradict \ref{fac2}.
\end{proof}

\begin{lemma}\label{More1} Set $D:= \ker \big( \Id - R\big) = \Fix R $. Then the following hold:
	\begin{enumerate}	
		\item\label{fix1} $D$ is a closed linear subspace.
		\item\label{fix2} $\fix R^* = D $.
		\item\label{fix3} $\overline{\ran} \big( \Id - R\big) = \overline{\ran} \big( \Id - R^{*}\big) = D^{\perp}$.
	\end{enumerate}	
\end{lemma}

\begin{proof}
	\cref{fix1}: Let $x, y \in D$ such that $x - Rx = 0$ and $y - Ry = 0$. Let $\alpha, \beta \in \RR$. Then 
	\begin{align*}  
	\big( \Id - R \big) \big( \alpha x + \beta y \big) & = \big( \Id - R \big) ( \alpha x ) + \big( \Id - R \big) (\beta y ) \\
	& = \alpha ( x - Rx) + \beta ( y - Ry) \\
	& = 0 + 0 =0.
	\end{align*}
	Therefore, $\alpha x + \beta y \in D$ and hence $D$ is a linear subspace. To show that $D$ is closed, let $ (x_n)$ be a sequence in $D$ such that $(x_n)$ converges to $x$. Then 
	\begin{align*} 
	\lim\limits_{n \to \infty} ( \Id -R) (x - x_n) & = 	\lim\limits_{n \to \infty} ( \Id - R) x - 	\lim\limits_{n \to \infty} ( \Id - R) x_n \\
	& = ( \Id - R) x - ( \Id - R) x = 0.
	\end{align*}
	Therefore, $ x \in D$ and hence $D$ is closed.\\
	\cref{fix2} and \ref{fix3}: It follows from \cref{propr}\ref{fac2}~\&~\ref{fac3} that $ \Id - R$ is monotone and bounded. Hence, $\fix R^* = \fix R= D $ and $\overline{\ran} \big( \Id - R\big) = \overline{\ran} \big( \Id - R^{*}\big) = D^{\perp}$ by \cite[Proposition~20.17]{BC2017}.
\end{proof}

\begin{remark}\label{More2}Suppose that $X = \ell_2 (\NN)$ and that 
	\begin{equation}
	R:X\to X:(x_n)_{n \in \NN} \mapsto \big( ((1 - \eps_n) x_n)\big)_{n\in \NN},
	\end{equation}
	where $(\eps_n)_{n \in \NN}$ lies in $\left] 0 , 1 \right[$ with $\eps_n \to 0$. Then the following hold:
	\begin{enumerate}
		\item\label{ran1} $\Id -R: (x_n)_{n \in \NN} \mapsto \big( \eps_n x_n\big)_{n \in \NN}$ is a compact operator. 
		\item\label{ran2} $D = \Fix R= \{0\}$.
		\item\label{ran3} $\ran (\Id - R)$ is not closed.
		\item\label{ran4} $\ran R$ is a closed subspace.
	\end{enumerate}
\end{remark}

\begin{proof}
	\cref{ran1} and \cref{ran2}: See \cite[PropositionII.4.6]{conway2019course}.
	\cref{ran3}:  It follows from \cite[Proposition~3.4.6]{megginson2012introduction} that $\ran (\Id - R)$ is closed if and only if $\ran (\Id - R)
	$ is finite-dimensional. On the other hand, $X = D^{\perp} = \overline{\ran} (\Id - R)$, i.e., the range of $\Id - R$ is dense in the infinite-dimensional space $X$. Altogether, 
	$$  \ran \big( \Id - R\big) \ \text{is not closed}.$$
	\cref{ran4}: See \cite[Lemma 3.4.20]{megginson2012introduction}.		
\end{proof}

\begin{lemma}\label{T-opr} Suppose that $ \ran ( \Id - R)$ is closed; equivalently,
\begin{empheq}[box=\mybluebox]{equation*} 
{ \ran ( \Id - R) = D^\perp.}
\end{empheq}	
Set 
\begin{empheq}[box=\mybluebox]{equation}\label{Top} 
	{T:= \PR_{D^\perp} ( \Id - R)^{-1} \PR_{D^\perp} - \frac{1}{2} \PR_{D^\perp}.}
\end{empheq} 
	Then, 
	\begin{enumerate}
		\item\label{T-opr00} $\ran ( \Id - R)^{*} = D^{\perp}$.
		\item\label{T-opr01} $T$ is a linear  and continuous.
		\item\label{T-opr001} $T$ is monotone.
		\item\label{T-opr02} $T$ is maximally monotone.
		\item\label{T-opr03} $ \ran T \subseteq  D^{\perp}$, where $D = \ker ( \Id - R)$.
		\item\label{T-opr04} $ \PR_{D^\perp} T = T \PR_{D^\perp} = T$.	
	\end{enumerate}
\end{lemma}

\begin{proof}
\ref{T-opr00}: By using the closeness of $ \ran (\Id - R)$ and ... 
\begin{align*}
\ran ( \Id - R)^{*} & = \ran ( \Id - R^{*}) \\
& = \ran ( \Id - R) \\
& = D^{\perp}.
\end{align*}
\ref{T-opr01}: This is clear because $T$ is defined using $\PR_{D^\perp}$, which is a linear and continuous operator.
\ref{T-opr001}: See \cite[Example~20.12]{BC2017}. \ref{T-opr02}: Combine \ref{T-opr01}, \ref{T-opr001} and \cite[Corollary~20.28]{BC2017}.
\ref{T-opr03}: It follows directly from \eqref{Top}.
\ref{T-opr04}: Since $ \ran T \subseteq D^{ \perp}$ by using \ref{T-opr03}, we obtain 
$$  \PR_{D^\perp} T = T.  $$
Moreover, both $T$ and $\PR_{D^\perp}$ commute and so 
$$  T  \PR_{D^\perp} = \PR_{D^\perp} T = T.$$
\end{proof}

\begin{remark} It is well known that $ \ran (\Id -R)$ is closed if and only if there exists $\alpha > 0$ such that 
	\begin{equation}\label{done1}
	\Big( \forall y \in ( \ker (\Id - R))^\perp = D^\perp \Big) \ \ \norm{ y - Ry} \geq \alpha \norm{y};
	\end{equation}
\end{remark}
\begin{proof}
See \cite[Theorem~8.18]{deutsch2001best}.
\end{proof}

\begin{proposition}\label{Prro} Suppose that \cref{done1} holds, then the operator 
\begin{empheq}[box=\mybluebox]{equation}\label{A-Aopr} {\PR_{D^\perp}( \Id - R)^{-1} : D^\perp \to D^\perp,}
\end{empheq}
\begin{enumerate}
\item\label{prp1} is a linear selection of $T^{-1}$. 
\item\label{prp2} is continuous and its norm is bounded above by $1 / \alpha$.	
\end{enumerate}	
\end{proposition}
\begin{proof}
\ref{prp1}: It follows from \cref{Top} and \cref{T-opr}.
\ref{prp2}: Clear from \cref{done1}.
\end{proof}
\begin{theorem}\label{the1} Suppose that $\ran ( \Id - R)$ is closed. Set 
\begin{empheq}[box=\mybluebox]{equation}\label{A-opr} {A:= ( \Id - R)^{-1} - \frac{1}{2} \Id,}
\end{empheq}	
and defined 
\begin{equation}
Q_A: \dom A \to X: y \mapsto \PR_{Ay}y.
\end{equation}
Set
\begin{empheq}[box=\mybluebox]{equation}\label{B-opr} {B:= \PR_{ \dom A} Q_A \PR_{ \dom A}.}
\end{empheq}
Then the following holds;
\begin{enumerate}
\item\label{the102} $ \dom A = D^{\perp}$ and is closed.
\item\label{the101} $A$ is linear relation.
\item\label{the103}  $A$ is maximally monotone.
\item\label{the104}  we have 
\begin{equation*}
( \forall y \in \dom A) \ \ Q_A y = \PR_{D^\perp} ( \Id - R)^{-1} y - \frac{1}{2} \PR_{D^\perp} y.
\end{equation*}
\item\label{the105} $B$ is maximally monotone, linear and continuous.
\item\label{the107}  $A= N_{D^\perp} + B$.
\item\label{the108}  $B = T$.
\item\label{the1007} $\left.B\right|_{\dom A}$ is a selection of $\left.A\right|_{\dom A} $.  	
\end{enumerate} 
\end{theorem}
\begin{proof}
\ref{the102}: From \cref{A-opr} $ \dom A = \ran (\Id - R) = D^{\perp}$, which is closed by the assumption.\\
\ref{the101}: It is clear that $A$ is a linear relation, i.e., $\gra A$ is a linear subspace, that $A0 = D$, and by \ref{the102} the $\dom A = D^{\perp}$ is closed. \\
\ref{the103}: It follows directly from \cref{propr}\ref{fac8}.\\
\ref{the104}: By \cite[Proposition~6.2]{bauschke2010borwein}, we have $ \big( \forall y \in \dom A \big) \ Q_A y = \PR_{ (A0)^{\perp}} (Ay) \in Ay$. Hence, 
\begin{align*}
\big( \forall y \in D^{\perp}\big) \ \ Q_Ay & = \PR_{D^\perp} ( Ay) = \PR_{D^\perp} \Big( ( \Id - R)^{-1} y - \frac{1}{2} y \Big) \\
& = \PR_{D^\perp} ( \Id - R)^{-1} y - \frac{1}{2} \PR_{D^\perp} y.
\end{align*}
\ref{the105}: See \cite[Example~6.4(i)]{bauschke2010borwein}.
\ref{the107}: Combining \ref{the102} and \cite[Example~6.4(iii)]{bauschke2010borwein} gives 
$$  A = N_{\dom A} + B = N_{D^\perp} + B.$$
\ref{the108}: Using \eqref{B-opr}, \ref{the104} and \ref{the102} gives 
\begin{align*}
B & = \PR_{ \dom A} Q_A \PR_{ \dom A} \\
& = \PR_{D^\perp} \Big( \PR_{D^\perp} ( \Id - R)^{-1} - \frac{1}{2} \PR_{D^\perp} \Big) \PR_{D^\perp} \\
& = \PR_{D^\perp} ( \Id - R)^{-1} \PR_{D^\perp} - \frac{1}{2} \PR_{D^\perp} \\
& = T \ \ \ \ \ (\text{from} \ \cref{Top}).
\end{align*}
\ref{the1007}: Using \ref{the102} gives 
\begin{align*}
\left.A\right|_{\dom A} & =   \left.( N_{D^\perp} + B )\right|_{D^\perp} \ \ ( \text{from \ref{the107}} )  \\
& = \left.( N_{D^\perp} +  T )\right|_{D^\perp} \ \ ( \text{from \ref{the108}} ) \\
& = \left.\Big( N_{D^\perp} +  \PR_{D^\perp} ( \Id - R)^{-1} \PR_{D^\perp} - \frac{1}{2} \PR_{D^\perp}  \Big)\right|_{D^\perp} \ \ ( \text{from} \ \cref{Top}) \\
& \equiv D+ D^{\perp} \ \ (\text{because} \ \left.N_{D^\perp}\right|_{D^\perp} \equiv D ),
\end{align*}
and 
\begin{align*} 
\left.B\right|_{\dom A} & = \left.T\right|_{D^\perp} \ \ ( \text{from \ref{the102} and \ref{the108}}) \\
& = \left.\Big(  \PR_{D^\perp} ( \Id - R)^{-1} \PR_{D^\perp} - \frac{1}{2} \PR_{D^\perp}  \Big)\right|_{D^\perp} \ \ ( \text{from \cref{Top}}) \\
& = D^{\perp}.
\end{align*} 
Hence, $\left.B\right|_{\dom A}$ is a selection of $\left.A\right|_{\dom A} $.
\end{proof}

In the next theorem we derive formulas for the inverse and Moore-Penrose inverse of the operator $ ( \Id - R)$.
\begin{theorem}\label{The11} Recall from \cref{Top} and \cref{A-opr} that 
$$	T:= \PR_{D^\perp} ( \Id - R)^{-1} \PR_{D^\perp} - \frac{1}{2} \PR_{D^\perp},$$
and 
$$  A:= ( \Id - R)^{-1} - \frac{1}{2} \Id,  $$
respectively.Then the following hold;
\begin{enumerate} 
\item\label{The110} The set-valued inverse of $ \Id - R$ is 
\begin{empheq}[box=\mybluebox]{equation}\label{Inv} 
{( \Id - R)^{-1} = \frac{1}{2} \Id + T + N_{D^{\perp}}.}
\end{empheq}
\item\label{The111} The Moore-Penrose inverse of $ \Id - R$ is 
\begin{empheq}[box=\mybluebox]{equation}\label{InvMo} {( \Id - R)^{\dagger} =  T + \frac{1}{2} \PR_{D^\perp}.}
\end{empheq}
\end{enumerate} 
\end{theorem}  
\begin{proof}
\ref{The110}: Combining \cref{the1}\ref{the107}~\&~\ref{the108} and \cref{A-opr} gives 
\begin{align*}
( \Id - R)^{-1} - \frac{1}{2} \Id & = A \\
& = N_{D^\perp} + T,
\end{align*}
Hence, 
$$ ( \Id - R)^{-1} = N_{D^\perp} + T + \frac{1}{2} \Id.  $$
\ref{The111}: By using \cite[Proposition~2.1]{bauschke2012compositions} and we obtain 
\begin{align*}
( \Id - R)^{\dagger} & = \PR_{ ( \Id - R)^*} \circ ( \Id - R)^{-1} \circ \PR_{ \ran ( \Id - R)} \\
& = \PR_{D^\perp} \circ ( \Id - R)^{-1} \circ \PR_{ D^\perp}  \ \ ( \text{from \cref{T-opr}\ref{T-opr00}})\\
& = \PR_{D^\perp} \circ \Big(  \frac{1}{2} \Id + T + N_{D^{\perp}} \Big) \circ \PR_{ D^\perp} \ \ ( \text{from \ref{The110}}) \\
& = \PR_{D^\perp} \circ \Big( \frac{1}{2} \PR_{D^\perp} + T \PR_{ D^\perp} + D \Big) \ \ (\text{Because $\left.N_{D^\perp}\right|_{D^\perp} \equiv D$}) \\
& = \frac{1}{2} \PR_{D^\perp} + \PR_{D^\perp} T \PR_{ D^\perp} + 0 \\
& = \frac{1}{2} \PR_{D^\perp} + T \ \ (\text{from \cref{T-opr}\ref{T-opr04}}),
\end{align*}
which verified \cref{InvMo}.
\end{proof}

\begin{proposition}[uniqueness of $T$] Let $T_{ \circ}: X \to X$ be such that
\begin{equation}\label{T0}
( \Id - R)^{-1} = \frac{1}{2}  \Id + T_{\circ} + N_{D^{\perp}},
\end{equation}
 and 
 \begin{equation}\label{T00}
 \PR_{D^\perp} T_{\circ} \PR_{D^{\perp}} = T_{\circ}.
 \end{equation}
 Then $  T_{\circ}  = T$.
\end{proposition}

\begin{proof}
By using \cref{Top}, we have 
\begin{align*}
T & = \PR_{D^\perp} ( \Id - R)^{-1} \PR_{D^\perp} - \frac{1}{2} \PR_{D^\perp} \\
& =  \PR_{D^\perp} \Big(  \frac{1}{2} \Id + T_{\circ}  + N_{D^\perp}\Big) \PR_{D^\perp} - \frac{1}{2} \PR_{D^\perp} \ \ ( \text{from \cref{T0}})\\ 
& = \PR_{D^\perp} T_{\circ} \PR_{D^\perp}\\
& = T_{\circ} \ \ ( \text{from \cref{T00}}),
\end{align*}
as claimed.
\end{proof}

\begin{theorem}\label{InT} Recall from \cref{Top} that 
$$ T =   \PR_{D^\perp} ( \Id - R)^{-1} \PR_{D^\perp} - \frac{1}{2} \PR_{D^\perp} . $$
Then the following hold;
\begin{enumerate}
\item\label{InT1}  $ ( 1/ 2) \Id + T $ is $\frac{1}{2}$-strongly monotone.
\item\label{InT0} $\big( ( 1/ 2) \Id + T\big)^{-1} = 2 J_{2T}.$
\item\label{InT2} $2T + \Id = 2 \PR_{D^\perp} (\Id - R)^{-1} \PR_{D^\perp} + \PR_{D}$.
\item\label{InT201} $J_{2T} = \PR_{D} + \frac{1}{2} ( \Id - R) \PR_{D^\perp}$.
\item\label{InT3} $2J_{2 T} = ( \Id - R) \PR_{D^\perp} + 2 \PR_{D}$.
\item\label{InT4} $( \Id - R) \PR_{D^\perp} + 2 \PR_{D} = \Id - R + 2 \PR_{D}$.
\item\label{InT5} We have 
\begin{empheq}[box=\mybluebox]{equation}\label{T-inver} { \Big( \frac{1}{2} \Id + T \Big)^{-1}=  2J_{2T} = ( \Id - R) \PR_{D^\perp }+ 2 \PR_{D}= \Id - R + 2 \PR_{D}.}
\end{empheq}
\item\label{InT6} $ \left.\Big(  \frac{1}{2}  \Id + T\Big)^{-1}\right|_{D^\perp} = \Id - R.$
\end{enumerate}
\end{theorem}
\begin{proof} 
\ref{InT1}: Showing that $ \frac{1}{2} \Id + T $ is $(1/2)$-strongly monotone  $  \Leftrightarrow \frac{1}{2} \Id + T - \frac{1}{2} \Id = T$ is montone, which is verified by \cref{T-opr}\ref{T-opr001}.
\ref{InT0}: From \cref{T-opr}\ref{T-opr01}~\&~\ref{T-opr02} and \cite[Lemma~2]{eckstein1992douglas}, we have 
\begin{align*}
\Big(  \frac{1}{2}  \Id + T \Big)^{-1} & = \Big( \frac{1}{2}  ( \Id + 2T)\Big)^{-1} \\
& = 2  \big(  \Id + 2T)^{-1} \\
& = 2  J_{2T} .
\end{align*}
\ref{InT2}: By using \cref{Top} and \cref{T-opr}\ref{T-opr01}, we obtain 
\begin{align*}
2T& = 2 \Big(    \PR_{D^\perp} ( \Id - R)^{-1} \PR_{D^\perp} - \frac{1}{2} \PR_{D^\perp}\Big) \\
& = 2  \PR_{D^\perp} ( \Id - R)^{-1} \PR_{D^\perp} - \PR_{D^\perp},
\end{align*}
hence
\begin{align*}
2T+ \Id & = 2  \PR_{D^\perp} ( \Id - R)^{-1} \PR_{D^\perp} - \PR_{D^\perp} + \Id \\
& =  2  \PR_{D^\perp} ( \Id - R)^{-1} \PR_{D^\perp} + \PR_{D}.
\end{align*}
\ref{InT201}: From \ref{InT2}, we obtain $2T+ \Id = 2  \PR_{D^\perp} ( \Id - R)^{-1} \PR_{D^\perp} + \PR_{D}$. Put differently, 
$$  2T + \Id: D \oplus D^\perp \to D \oplus D^\perp : d \oplus d^\perp \mapsto d + 2 \PR_{D^\perp} ( \Id - R)^{-1} d^\perp. $$
For two vectors $d^\perp$, $e^{\perp}$ in $D^\perp$, we have the equivalences, 
\begin{equation*}
e^\perp = 2 \PR_{D^\perp} ( \Id - R)^{-1} d^{\perp}  \Leftrightarrow d^\perp = \Big(    2 \PR_{D^\perp} ( \Id - R)^{-1} \Big)^{-1} e^\perp ,
 \end{equation*}
 and therefore, 
 \begin{align*}
 d^\perp & = \Big(    2 \PR_{D^\perp} ( \Id - R)^{-1} \Big)^{-1} e^\perp \\
 & = \Big(    2 \big(  \PR_{D^\perp} ( \Id - R)^{-1} \big) \Big)^{-1} e^\perp \\
 & = \frac{1}{2} \Big(  \PR_{D^\perp} ( \Id - R)^{-1} \Big)^{-1} e^\perp \\ 
 & = \frac{1}{2}  \big(   \Id - R \big) \PR^{-1}_{D^\perp} e^\perp \\
 & = \frac{1}{2}  \big(   \Id - R \big) e^\perp.
 \end{align*}
 Hence, 
 \begin{equation*}
 ( 2T+ \Id)^{-1}:  D \oplus D^{\perp} \to D \oplus D^{\perp}: d \oplus d^\perp \mapsto d + \frac{1}{2} ( \Id - R) d^\perp;
 \end{equation*}
 equivalently, 
 \begin{equation*}
 J_{2T} = (2T + \Id)^{-1}: z \mapsto \PR_{D} z + \frac{1}{2} ( \Id - R) \PR_{D^\perp} z.
\end{equation*}
 \ref{InT3}: It follows directly from \ref{InT201}. \ref{InT4}: Because $ \ker ( \Id - R) = D$, we have $ ( \Id - R) \PR_{D} \equiv 0$. Therefore,
 $$   ( \Id - R) \PR_{D^\perp} + 2 \PR_{D} = \Id - R + 2 \PR_{D}.$$
 \ref{InT5}: Combine \ref{InT0}, \ref{InT3}, and \ref{InT4}. \ref{InT6}: From \ref{InT3}, we obtain 
\begin{align*} 
  \left.\Big(  \frac{1}{2}  \Id + T\Big)^{-1}\right|_{D^\perp} & =   \left.\big(  \Id - R + 2 \PR_{D} \big)\right|_{D^\perp} =  \Id - R.
\end{align*} 
\end{proof}

\begin{proposition} Let $ m \in \{ 2, 3, \dots \}$ and assume that $R^{m} = \Id$, i.e., $R$ is an isometry of finite rank $m$.  Assume that $X = R^{m}$ and recall from \cite[Lemma]{Alwadanithesis} that 
\begin{equation}\label{proj}
\PR_{D} = \frac{1}{m}  \sum_{k=0}^{m-1} R^k \ \ \ \ \text{and} \ \ \ \ \PR_{D^\perp} = \Id - \frac{1}{m}  \sum_{k=0}^{m-1} R^k,
\end{equation}
where $D= \fix R$. Then 
\begin{empheq}[box=\mybluebox]{equation}\label{R-s} { \frac{1}{2} \PR_{D^\perp} \big( R+ R^{*} \big) \PR_{D^\perp} = \frac{1}{m} \Big(    - \Id - \sum_{k=2}^{m-2} R^k + \frac{ \max \{ 1, m-2 \}}{2} \big( R+ R^{m-1} \big)\Big)}.
\end{empheq}
\end{proposition}

\begin{proof}
Noted that $R$ is an isometry $\Rightarrow R^{*} R = R R^{*} = \Id$, so $ R^{-1} = R^{*}$. But also $R$ has rank $m$, hence 
$R^{m-1} = R^{-1} = R^{*}$.  By using these facts, we obtain 
\begin{align*}
\PR_{D^\perp} ( R+ R^{*}) \PR_{D^\perp} & = \PR_{D^\perp}  \big( R+ R^{-1} \big) \PR_{D^\perp} \\
& =  \PR_{D^\perp}  \big( R+ R^{-1} \big)  \Big( \Id - \frac{1}{m}  \sum_{k=0}^{m-1} R^k \Big) \ \ \ \ ( \text{from \cref{proj}}) \\
& = \PR_{D^\perp} \Big(  \big( R+ R^{-1} \big) - \frac{1}{m}   \sum_{k=0}^{m-1}  \big( R+ R^{-1} \big) R^{k} \Big) \\
& = \PR_{D^\perp} \Big(  \big( R+ R^{-1} \big) - \frac{1}{m} \sum_{k=0}^{m-1}   \big( R^{k+1} + R^{k-1} \big) \Big).
\end{align*}
Since $R$ has rank $m$, the following holds:
\begin{equation}\label{Need1}
\sum_{k=0}^{m-1}  R^{k+1} = \sum_{k=0}^{m-1} R^{k-1} = \sum_{k=0}^{m-1} R^{k}.
\end{equation}
Moreover,
\begin{equation}\label{Need2}
R^{l} \sum_{k=0}^{m-1} R^{k} = \sum_{k=0}^{m-1} R^{l+k} = \sum_{k=0}^{m-1} R^{k}
\end{equation}
Thus, 
\begin{equation}\label{Need3}
\big(   R+ R^{-1} \big)  \Big(    \frac{1}{m} \sum_{k=0}^{m-1} R^{k} \Big) = \frac{1}{m} \sum_{k=0}^{m-1} \big(  R^{k+1} + R^{k-1} \big) = \frac{2}{m} \sum_{k=0}^{m-1} R^{k}.
\end{equation}
Therefore, 
\begin{align*}
\PR_{D^\perp} \big( R+ R^* \big) \PR_{D^\perp} & = \PR_{D^\perp}  \Biggl(    \big(   R+ R^{-1}\big)   - \frac{1}{m} \sum_{k=0}^{m-1}  \big(    R^{k+1} + R^{k-1}\big)   \Biggl) \\
& = \PR_{D^\perp} \Biggl(   \big(  R+ R^{-1} \big)   - \frac{2}{m}  \sum_{k=0}^{m-1}  R^k  \Biggl)  \\
& = \Biggl(  \Id - \frac{1}{m}  \sum_{k=0}^{m-1}  R^k  \Biggl) \Biggl(  \big(    R+ R^{-1}\big)  - \frac{2}{m}    \sum_{k=0}^{m-1} R^k \Biggl) \\
& = \Biggl(   \big(   R+ R^{-1}\big)  - \frac{2}{m}   \sum_{k=0}^{m-1} R^{k} \Biggl) - \Biggl(   \frac{1}{m}  \sum_{k=0}^{m-1} R^{k}\Biggl) \Biggl(  \big(    R+ R^{-1}\big)  - \frac{2}{m}   \sum_{k=0}^{m-1} R^k  \Biggl) \\
& = \Biggl(   \big(   R+ R^{-1}\big)  - \frac{2}{m}  \sum_{k=0}^{m-1} R^{k} \Biggl)- \Biggl(   \frac{2}{m} \sum_{k=0}^{m-1} R^k - \frac{2}{m^2}    \sum_{l=0}^{m-1} R^l  \sum_{k=0}^{m-1} R^k \Biggl) \\
& = \Biggl(    \big(    R+R^{-1}\big)  - \frac{2}{m}    \sum_{k=0}^{m-1}  R^k \Biggl)  - \Biggl(  \frac{2}{m}   \sum_{k=0}^{m-1} R^k - \frac{2}{m^2}     \sum_{l=0}^{m-1}  \sum_{k=0}^{m-1} R^k\Biggl) \\
& = \Biggl(   \big(  R+ R^{-1} \big) - \frac{2}{m}   \sum_{k=0}^{m-1}  R^k  \Biggl) - \Biggl(   \frac{2}{m}   \sum_{k=0}^{m-1}  R^k - \frac{2m}{m^2}   \sum_{k=0}^{m-1}  R^k \Biggl) \\
& = \Biggl(    \big(   R+ R^{-1} \big)  - \frac{2}{m}   \sum_{k=0}^{m-1}  R^k  \Biggl) - \Biggl(    \frac{2}{m}    \sum_{k=0}^{m-1}  R^k -  \frac{2}{m}    \sum_{k=0}^{m-1}  R^k  \Biggl) \\
& = \Biggl(    \big(  R+ R^{-1} \big)  - \frac{2}{m}   \sum_{k=0}^{m-1}  R^k \Biggl).
\end{align*}
First: assume that $m > 2$. Therefore, $ \max \{ 1, m-2 \} = m - 2$. Then 
\begin{align*}
\PR_{D^\perp} \big( R+ R^* \big) \PR_{D^\perp} & = \big(  R+ R^{-1} \big) - \frac{2}{m} \sum_{k=0}^{m-1}  R^k \\
& = \frac{2}{m} \Biggl(   \frac{m}{2}   \big(  R+ R^{-1}\big)   -   \sum_{k=0}^{m-1}  R^k  \Biggl) \\
& = \frac{2}{m}  \Biggl(    \Biggl(   \frac{m}{2} - 1\Biggl)   \big( R+ R^{-1} \big) - \Id -  \sum_{k=2}^{m-2} R^k \Biggl) \\
& = \frac{2}{m}  \Biggl(   - \Id + \frac{m-2}{2}   \big(  R+ R^{m-1}\big)  -  \sum_{k=2}^{m-2} R^k \Biggl),  
\end{align*}
which prove \cref{R-s} when $m > 2$. \\
Next, assume that $m = 2$. Then $\max \{ 1, m-1  \} = 1$ and $R^{-1} = R^{2-1} = R$. Therefore, 
\begin{align*}
\PR_{D^\perp} \big( R+ R^* \big) \PR_{D^\perp} & = \big( R+ R^{-1}\big) - \frac{2}{m}   \sum_{k=0}^{m-1}  R^k \\
& = 2R- \frac{2}{2} \big(   \Id + R \big) \\
& = 2R-\Id - R \\
& = R- \Id.
\end{align*}
On the other hand, 
\begin{align*}
\frac{2}{m} \Biggl(  - \Id + \frac{ \max \{ 1, m-2 \}}{2}  \big(   R+ R^{m-1}\big) -   \sum_{k=2}^{m-2}  R^k \Biggl) & = \frac{2}{2}  \Biggl(   - \Id + \frac{1}{2}  \big(  R+ R \big) -  \sum_{k=2}^{0}  R^k \Biggl) \\
& = - \Id + \frac{1}{2} ( 2R) - 0 \\
& = - \Id + R,
\end{align*}
so equality holds when $m =2$.
\end{proof}
\section{Examples}
\begin{example}[isometry of finite rank] Let $m \in \{ 2, 3, \dots \}$ and assume that 
\begin{equation}\label{isom-R}
R^m = \Id.
\end{equation}
Then the results in \cref{Section2} were derived already in \cite{alwadani2021resolvents}. Moreover, the work there based on exploiting \cref{isom-R} yielded to \cref{proj} and 
\begin{equation}\label{pps}
T= \frac{1}{2m}   \sum_{k=1}^{m-1} \big( m - 2k \big)R^k = -T^*,
\end{equation}
which is always skew right-shift operator, $T$ is symmetric only when $m = 2$.
\end{example}

\begin{example}\label{examp1} Let $U$ be a closed subspace of $X$ and suppose that 
\begin{equation}\label{U}
R= \PR_{U}.
\end{equation}
Then 
\begin{enumerate}
\item\label{Ex1} $D = U$.
\item\label{Ex2} $ \Id - R = \PR_{U^\perp}$.
\item\label{Ex3} $ \ran (\Id - R) = D^\perp$ is closed.
\item\label{Ex4} $ \big(  \Id - R \big)^{-1} = \Id + N_{U}$.
\item\label{Ex5} $T = \frac{1}{2} \PR_{U^\perp} = T^*$.
\item\label{Ex6} $T$ is always symmetric, but skew only when $U = X$.
\end{enumerate}
\end{example}

\begin{proof}
\ref{Ex1}: $ D = \fix R = \fix \PR_{U} = \{ x \in X \mid x = \PR_{U} x\} = U$. \ref{Ex2}: $ \Id - R = \Id - \PR_{U} = \PR_{U^\perp}$. \ref{Ex3}: By using \ref{Ex2}, we obtain $ \ran ( \Id - R) = \ran ( \Id - \PR_{U}) = U^{\perp} = D^\perp$. \ref{Ex4}: From \cite[Example~1]{BC2017}, we have $\big(  \Id - R \big)^{-1}  = \big(  \Id - \PR_{U} \big)^{-1} =  \PR^{-1}_{U^\perp} = \Id + N_{U^\perp}$. \\
\ref{Ex5}: By using \cref{Top}, we have 
\begin{align*}
T & = \PR_{D^\perp} ( \Id - R)^{-1} \PR_{D^\perp} - \frac{1}{2} \PR_{D^\perp} \\
& = \PR_{U^\perp} ( \Id + N_{U^\perp}) \PR_{U^\perp} - \frac{1}{2} \PR_{U^\perp}\\
& =  \frac{1}{2} \PR_{U^\perp} \\
& = T^*.
\end{align*}
\ref{Ex6}: Follows from \ref{Ex5}.
\end{proof}

\begin{example}\label{examp2} Let $U$ be a closed subspace of $X$ and suppose that 
\begin{equation}\label{min-U}
R= - \PR_{U}.
\end{equation}
Then 
\begin{enumerate}
\item\label{Ux1} $D = \{ 0 \}$.
\item\label{Ux2} $ \Id - R = \Id + \PR_{U}$.
\item\label{Ux3} $ \ran (\Id - R) = X$ .
\item\label{Ux4} $ \big(  \Id - R \big)^{-1} = \frac{1}{2} \Id + \frac{1}{2} \PR_{U^\perp}$.
\item\label{Ux5} $T = \frac{1}{2} \PR_{U} $.
\end{enumerate}
\end{example}

\begin{proof}
\ref{Ux1}: $ D = \fix R = \fix (-\PR_{U} )= \{ x \in X \mid x = -\PR_{U} x\} =  \{ 0 \}$. \ref{Ux2}:  $ \Id - R = \Id +  \PR_{U}$.
\ref{Ux3}: By \cite[Minty Theorem]{BC2017}, $Id +  \PR_{U} $ has full range $D^{}= X$.
\ref{Ux4}: $\big(  \Id - R \big)^{-1}  = J_{\PR_{U}}  = \frac{1}{2} \PR_{U} + \PR_{U^\perp} = \frac{1}{2} \Id + \frac{1}{2} \PR_{U^\perp}$.
\ref{Ux5}: We have 
\begin{align*}
 T & =  \PR_{D^\perp} ( \Id - R)^{-1} \PR_{D^\perp} - \frac{1}{2} \PR_{D^\perp} \\
 & = \frac{1}{2} \Id  + \frac{1}{2} \Id - \frac{1}{2} \PR_{U^\perp} - \frac{1}{2} \Id \\
 & = \frac{1}{2} \PR_{U}.
\end{align*}
\end{proof}

\begin{example}\label{examp3} Let $U$ be a closed subspace of $X$ and suppose that 
\begin{equation}\label{RU}
R= R_{U}.
\end{equation}
Then 
\begin{enumerate}
\item\label{RRx1} $D = U$.
\item\label{RRx2} $ \Id - R = 2 \PR_{U^\perp}$.
\item\label{RRx3} $ \ran (\Id - R) = D^\perp$ is closed.
\item\label{RRx4} $ \big(  \Id - R \big)^{-1} = \frac{1}{2} \Id + N_{U}$.
\item\label{RRx5} $T = 0$.
\end{enumerate}
\end{example}

\begin{proof}
\ref{RRx1}: $ D = \fix R = \fix (R_{U} )= \{ x \in X \mid x = R_{U} x \} =  \{ x \in X \mid 2 x = 2 \PR_{U} \} = U$. \\
\ref{RRx2}: $ \Id - R =  \Id - R_U = \big(   \PR_U + \PR_{U^\perp}\big) - \big(  \PR_U - \PR_{U^\perp}\big) = 2 \PR_{U^\perp}$.
\ref{RRx3}:  $\ran (\Id - R) = \ran (  2 \PR_{U^\perp} ) = D^\perp$ is closed. 
\ref{RRx4}: $ \big( \Id - R \big)^{-1} = \big(  2 ( \Id  -  \PR_{U})\big)^{-1}  = \frac{1}{2} \Id + N_{U^\perp}$.
\ref{RRx5}: We have 
\begin{align*}
T & =  \PR_{D^\perp} ( \Id - R)^{-1} \PR_{D^\perp} - \frac{1}{2} \PR_{D^\perp} \\ 
& =  \PR_{D^\perp} \Big(   \frac{1}{2} \Id +N_{U^\perp} \Big) \PR_{U^\perp} - \frac{1}{2} \PR_{U^\perp}\\
& = \frac{1}{2} \PR_{U^\perp} - \frac{1}{2} \PR_{U^\perp} \\
& = 0.
\end{align*}
\end{proof}

\begin{example}\label{examp4} Let $U$ be a closed subspace of $X$ and suppose that 
\begin{equation}\label{-RU}
R= -R_{U}.
\end{equation}
Then 
\begin{enumerate}
\item\label{IRx1} $D = \fix \big( - R_{U} \big)=U^{\perp}$.
\item\label{IRx2} $ \Id - R = 2 \PR_{U}$.
\item\label{IRx3} $ \ran (\Id - R) = U$ is closed.
\item\label{IRx4} $ \big(  \Id - R \big)^{-1} = \frac{1}{2} \Id + N_{U}$.
\item\label{IRx5} $T = 0$.
\end{enumerate}
\end{example}

\begin{proof}
\ref{IRx1}: Note that $- R_{U} = R_{U^\perp}$ and we learn from \cref{examp3} that $D= \fix R = U^{\perp}$.
\ref{IRx2}: $\Id - R = \Id - R_{U^\perp} = \big( \PR_U + \PR_{U^\perp}\big) - \big( 2 \PR_{U^\perp} - \Id \big) = \big( \PR_U + \PR_{U^\perp}\big) -  \big( \PR_{U^\perp} - \PR_U  \big) = 2 \PR_{U}$.
\ref{IRx3}:  By using \ref{IRx2}, we have $\ran \big(  \Id - R  \big) = \ran \big(  2 \PR_{U} \big) = D = U$.
\ref{IRx4}: $ \big( \Id - R \big)^{-1} = \big( \Id - (R_{U^\perp}) \big)^{-1} =\big(  \Id - (2 \PR_{U^{\perp}} - \Id )\big)^{-1} = \big( 2 ( \Id - \PR_{U^{\perp}}) \big)^{-1} = \frac{1}{2} \Id + N_{U}$ by \cite[Example]{BC2017}.
\ref{IRx5}: By using \cref{Top}, we have 
\begin{align*}
T & =  \PR_{D^\perp} ( \Id - R)^{-1} \PR_{D^\perp} - \frac{1}{2} \PR_{D^\perp} \\ 
& =  \PR_{D^\perp} \Big(   \frac{1}{2} \Id +N_{U} \Big) \PR_{U^\perp} - \frac{1}{2} \PR_{U^\perp}\\
& = \frac{1}{2} \PR_{U^\perp} - \frac{1}{2} \PR_{U^\perp} \\
& = 0.
\end{align*}
\end{proof}

\end{document}